\documentclass[a4paper,12pt,reqno]{amsart}
\usepackage[utf8]{inputenc}
\usepackage{amssymb}
\usepackage{amsmath}
\usepackage{amsthm}
\usepackage{enumerate}
\usepackage{hyperref}
\hypersetup{
	colorlinks=true,
	linkcolor=black,
	linkbordercolor={1 1 1},
	citecolor=black
}

\makeatletter
\@namedef{subjclassname@2020}{\textup{2020} Mathematics Subject Classification}
\makeatother

\newtheorem{thm}{Theorem}
\newtheorem{lem}[thm]{Lemma}
\newtheorem{cor}[thm]{Corollary}
\newtheorem{prop}[thm]{Proposition}
\newtheorem{proposition}[thm]{Proposition}

\theoremstyle{definition}
\newtheorem{rem}[thm]{Remark}

\newcommand{\cE}{\mathcal{E}}
\newcommand{\cK}{\mathcal{K}}
\newcommand{\cP}{\mathcal{P}}

\newcommand{\R}{\mathbb{R}}

\newcommand{\IB}{\mathbb{B}}

\newcommand{\IE}{\mathbb{E}}
\newcommand{\IH}{\mathbb{H}}

\newcommand{\IN}{\mathbb{N}}
\newcommand{\IP}{\mathbb{P}}

\newcommand{\IR}{\mathbb{R}}
\newcommand{\Sp}{\mathbb{S}}
\newcommand{\bS}{\mathbb{S}}

\newcommand{\bc}{\mathrm{bar}}
\newcommand{\Cov}{\mathrm{Cov}}
\DeclareMathOperator{\cl}{cl}

\DeclareMathOperator{\Int}{int}
\newcommand{\vol}{\mathrm{vol}}
\newcommand{\Zon}{\mathrm{Zon}}
\newcommand{\tr}{\mathrm{tr}}
\newcommand{\dd}{\mathrm{d}}
\newcommand{\Ent}{\mathrm{Ent}}

\title[Strange shadows of $\ell_p$-balls]{Strange shadows of $\ell_p$-balls}
\author{Zakhar Kabluchko}
\address{Zakhar Kabluchko, Institute of Mathematical Stochastics, University of M\"unster, Orl\'eans-Ring 10, 48149 M\"unster, Germany}

\email{zakhar.kabluchko@uni-muenster.de}

\author{Mathias Sonnleitner}
\address{Mathias Sonnleitner, Institute of Mathematical Stochastics, University of M\"unster, Orl\'eans-Ring 10, 48149 M\"unster, Germany and Department of Mathematical and Statistical Sciences, University of Alberta, Edmonton, AB, Canada, T6G 2G1}

\email{mathias.sonnleitner@ualberta.ca}

\date{\today}

\begin{document}

\begin{abstract}
	
	\noindent We prove a large deviations principle for orthogonal projections of the unit ball $\mathbb{B}_p^n$ of $\ell_p^n$ onto a random $k$-dimensional linear subspace of $\mathbb{R}^n$ as $n\to\infty$ in the case $2<p\le \infty$ and for the intersection of $\mathbb{B}_p^n$ with a random $k$-dimensional subspace in the case $1\le p <2$. The corresponding rate function is finite only on $L_q$-zonoids and their duals, respectively, and given in terms of the maximum entropy over suitable measures generating the $L_q$-zonoid, where $\frac{1}{p}+\frac{1}{q}=1$. In particular, we obtain that the renormalized projections/sections almost surely tend to a $k$-dimensional Euclidean ball of certain radius. Moreover, we identify the asymptotic probability that the random orthogonal projection remains within a ball of smaller radius. As a byproduct we obtain an interesting inequality for the Gamma function.
\end{abstract}

\maketitle

\section{Introduction and main results}

Let $\IB_p^n:=\{x\in\IR^n\colon \sum_{i=1}^{n}|x_i|^p\le 1\}$ be the unit ball of $\ell_p^n$, $1\le p<\infty$, and $\IB_{\infty}^n=[-1,1]^n$. Via orthogonal projection onto a $k$-dimensional subspace, $k\le n$, the $\ell_p$-ball $\IB_p^n$ casts a $k$-dimensional shadow. This terminology is also used in Ball's \textit{Shadows of Convex Bodies} \cite{Bal91}, where the author bounds the volume of a convex body by an average of volumes of some of its shadows. By a theorem of Dvoretzky-Milman~\cite[Cor.~7.24]{AS17}, as the dimension $n$ grows, most such shadows tend to be very close to the Euclidean ball $\IB_2^k$, after suitable renormalization. Thus, we might say that $\IB_2^k$ is the typical shadow of $\IB_p^n$. We are interested in shadows which are unlikely, \textit{strange} so to speak, occurring with exponentially small (in $n$) probability. This can be investigated within the framework of large deviations theory.

A classical result, which goes back to Maxwell, Borel and Poincar\'e, states that the coordinates of a random vector $X_n$ which is uniformly distributed in $n^{1/2}\IB_2^n$ become approximately Gaussian as $n\to\infty$. Due to \cite{RR91,SZ90} this also holds true for $n^{1/p}\IB_p^n$ but in this case the limit density is proportional to $e^{-|x|^p/p}$ for $p<\infty$ and the uniform density on $[-1,1]$ for $p=\infty$. In contrast, Klartag's central limit theorem for convex bodies \cite{Kla07} shows that for any fixed $k\in \IN$ most $k$-dimensional projections of $X_n$ become approximately Gaussian. Thus, in some sense, if $p\neq 2$, coordinate projections of $X_n$ are atypical. From a large deviations perspective this atypicality has been investigated in \cite{GKR16} and \cite{GKR17}, which motivated large deviation theorems for norms of $X_n$ in \cite{APT18,KPT19} and triggered a body of work, partly summarized in the surveys \cite{Pro24,PTT19}. For large deviations of random projections of the cube $\IB_{\infty}^n$, two more perspectives in terms of random measures have been considered in \cite{JKP22,KPT21}. On a related note, in \cite{PPZ14} a central limit theorem is proven for the volume of random projections of $\IB_{\infty}^n$. This has recently been extended to $p>1$ in \cite{PTT24} where also large deviation principles are proven. In the following, we provide yet another large deviation perspective of random projections of $\ell_p$-balls, which is in terms of the shape of the random projections of $\IB_p^n$ itself, rather than a distribution defined on it. This allows for a detailed description of the set of strange shadows of $\IB_p^n$.

More precisely, we are interested in large deviations of the sequence of random $k$-dimensional sets
\[
Z_{n,p}:=n^{1/p-1/2}\Pi_{n,k}^{\top}\mathbb{B}_p^n, \qquad n\ge k,
\]
where $k\in\IN$ is fixed and $\Pi_{n,k}$ is a random $n\times k$-matrix whose columns form an orthonormal $k$-frame uniformly distributed on the Stiefel manifold $\mathbb{V}_{n,k}$, the set of all such frames in $\IR^n$, with respect to Haar measure. That is, the first column $u_1$ of $\Pi_{n,k}$ is uniformly distributed on the unit sphere $\Sp^{n-1}$, the second column $u_2$ is uniformly distributed on $u_1^{\perp}\cap \Sp^{n-1}$, and so on. Then $E_{n,k}=\Pi_{n,k}(\IR^k)$ is a random $k$-dimensional subspace of $\IR^n$ uniformly distributed on the Grassmannian $\mathbb{G}_{n,k}$, the set of all such subspaces, with respect to Haar measure. Moreover, the composition of the isometry $\Pi_{n,k}^{\top}\colon E_{n,k}\to \IR^k$ and the orthogonal projection $P_{E_{n,k}}\colon \IR^n\to E_{n,k}$ onto $E_{n,k}$ satisfies $\Pi_{n,k}^{\top}P_{E_{n,k}}=\Pi_{n,k}^{\top}$. Thus, the action of $\Pi_{n,k}^{\top}$ can indeed be understood as taking orthogonal projection onto a random subspace and then mapping isometrically to $\IR^k$.  

The object $Z_{n,p}$ is a random element of the space $\mathcal{K}^k$ of convex bodies in $\IR^k$, that is, nonempty compact convex sets, which is equipped with the Hausdorff metric 
\[
\delta_H(K,L)
=\max\Big\{\sup_{x\in K}\inf_{y\in L}\|x-y\|,\sup_{x\in L}\inf_{y\in K}\|x-y\|\Big\},\quad K,L\in \cK^k.
\]
As a random linear image of $\IB_p^n$, the set $Z_{n,p}$ is a random compact set taking values in $\mathcal{K}^k$ in the sense of \cite{Mat75} and \cite{Mol05}. Note that for our purposes a convex body is a nonempty compact convex set which may have empty interior, see~\cite{Sch14}.

The following class of convex bodies will be essential for the study of the random convex body $Z_{n,p}$. Let $1<p\le \infty$ and let $1\le q<\infty$ be the Hölder-conjugate index with $\frac{1}{p}+\frac{1}{q}=1$, where $\frac{1}{\infty}=0$.  An $L_q$-zonoid (also sometimes called $L_q$-centroid body) is a convex body $K\in \cK^k$ with support function
\begin{equation} \label{eq:q-zonoid}
h(K,u)
:=\sup_{x\in K}\langle x,u\rangle
=\Big(\int_{\IR^k}|\langle x,u\rangle|^q \dd \mu(x)\Big)^{1/q},\qquad u\in \IR^k,
\end{equation}
for some Borel probability measure $\mu\in \cP_q(\IR^k)$ on $\IR^k$, where 
\[
\cP_q(\IR^k)
=\Big\{\mu\in \cP(\IR^k)\colon \int_{\IR^k}\|x\|_2^q\dd \mu(x)<\infty\Big\}
\]
and $\cP(\IR^k)$ is the set of all Borel probability measures on $\IR^k$. In this case, we write $K=Z_q(\mu)$ for the $L_q$-zonoid generated by $\mu$ and we denote by
\[
\Zon_{k,q}
=\{Z_q(\mu)\colon \mu\in\cP_q(\IR^k)\}
\]
the set of all $L_q$-zonoids in $\IR^k$. Then the random set $Z_{n,p}$ is a random $L_q$-zonoid, i.e. takes values in $\Zon_{k,q}$, see Proposition~\ref{pro:proj-zonoid} below. 

We will prove a large deviation principle (LDP) for $k$-dimensional random projections and sections of $\IB_p^n$ in the complete separable metric space $(\cK^k,\delta_H)$. Here, a sequence of random convex bodies $(K_n)_{n\in\IN}$ satisfies an LDP with lower semicontinuous rate function $I\colon \cK^k\to [0,\infty]$ if for all Borel sets $A\subset \cK^k$ it holds that
\begin{align*} 
-\inf_{K\in \Int A}I(K)
&\le \liminf_{n\to\infty} n^{-1} \log \IP[K_n\in A]\\
&\le \limsup_{n\to\infty} n^{-1} \log \IP[K_n\in A]
\le -\inf_{K\in \cl A}I(K),
\end{align*}
where $\Int A$ (respectively, $\cl A$) denotes the interior (respectively, closure) of $A$ and $\inf \emptyset=\infty$. A rate function is good if it additionally has compact sublevel sets. We refer to \cite{DZ10} for a background on the theory of large deviations. The rate functions in our main results involve the (differential) entropy which, for $\mu\in \cP(\IR^k)$, is given by
\begin{equation*} 
{\rm Ent}(\mu)
=	-\int_{\IR^k}\log(f(x))f(x)\dd x,
\end{equation*}
whenever $\mu$ admits a Lebesgue density $f$ and ${\rm Ent}(\mu)=-\infty$ otherwise. Our main result for random projections of $\IB_p^n$ reads as follows.

\begin{thm}\label{thm:proj}
	Let $2<p\le \infty$ and let $1\le q<2$ with $\frac{1}{p}+\frac{1}{q}=1$. Fix $k\in \IN$. The sequence $Z_{n,p}=n^{1/p-1/2}\Pi_{n,k}^{\top} \IB_p^n$, $n\ge k$, satisfies an LDP in $\cK^k$ with good rate function 
\[
I_p(K)
=\inf\{\Ent(\gamma^{\otimes k})-\Ent(\mu)\colon \mu\in\cP_q(\IR^k) \text{ with } Z_q(\mu)=K \text{ and } Z_2(\mu)\subset \IB_2^k\},
\]
where $K\in \cK^k$ and $\gamma^{\otimes k}$ denotes the $k$-dimensional standard Gaussian distribution. 
\end{thm}

Large deviation principles for Minkowski sums of i.i.d. random compact sets have been proven in \cite{Cer99}, and more recently in \cite{KM14,Mia09,MPS11}. In contrast, the setting studied here is related to $q$-sums of random convex sets, see Remark~\ref{rem:q-sum}. Our approach to prove Theorem~\ref{thm:proj} is based on a large deviation principle for the empirical measure of the rows of the random matrix $\Pi_{n,k}^{\top}$, recently established by Kim and Ramanan in \cite[Thm.~2.8]{KR23}.

\begin{rem}\label{rem:finite-rate}
Theorem~\ref{thm:proj} yields a description of the set of strange shadows. Let $2<p\le \infty$ and $K\in \cK^k$ with $0<I_p(K)<\infty$ with $I_p$ as in Theorem~\ref{thm:proj}. Then $K$ belongs to $\Zon_{k,q}$ and admits an absolutely continuous generating measure $\mu\in \cP_q(\IR^k)$ with $Z_q(\mu)=K$ and $Z_2(\mu)\subset \IB_2^k$. The set of all such $K$ can be interpreted as the set of all possible shadows of the ball $\IB_p^n$ that occur with probability exponentially small in $n$. (Shadows with $I_p(K)=\infty$ occur with probability super-exponentially small in $n$.) Further, the body $K$ has nonempty interior, see e.g.~Proposition~\ref{pro:interior} below. By Lyapunov's inequality, we have $Z_q(\mu)\subset Z_2(\mu)$ and thus $K\subset \IB_2^k$. Additionally, if $K=\alpha\IB_2^k$ for some $\alpha\in (0,1]$, then it holds that $\alpha\le a_{k,q}^{1/q}\sqrt{k}<1$, where $a_{k,q}$ is as in \eqref{eq:akq} in Section~\ref{sec:small-ball} below, see Remark~\ref{rem:akq}. 
\end{rem}

The geometric condition $Z_2(\mu)\subset \IB_2^k$ in Theorem~\ref{thm:proj} is equivalent to 
\begin{equation} \label{eq:covariance}
\Cov(\mu)
:=[\Cov_{X\sim \mu}(X_i,X_j)]_{i,j=1}^k
\le I_{k\times k},
\end{equation}
where $X\sim \mu$ means that $X$ has distribution $\mu$, $\le$ denotes the Loewner order on the cone of positive semi-definite matrices and $I_{k\times k}$ the identity matrix. It is well known that $\gamma^{\otimes k}$ is the unique maximum entropy distribution given \eqref{eq:covariance} and that $\Ent(\gamma^{\otimes k})=\frac{k}{2}\log(2\pi e)$, see also Section~\ref{sec:small-ball}. Therefore, for $K\in\cK^k$,
\begin{equation} \label{eq:alternative-rate}
I_p(K)
=\frac{k}{2}\log(2\pi e)-\sup\{\Ent(\mu)\colon \mu\in\cP_q(\IR^k) \text{ with } Z_q(\mu)=K, Z_2(\mu)\subset \IB_2^k\},
\end{equation}
and it holds that $I_p(K)\ge 0$ with $I_p(K)=0$ if and only if $K =Z_q(\gamma^{\otimes k})$. Rotational invariance of $\gamma^{\otimes k}$ implies that $Z_q(\gamma^{\otimes k})=m_q \IB_2^k$ with
\begin{equation} \label{eq:ball-min}
	m_{q}:=(\IE_{X\sim \gamma}|X|^{q})^{1/q}=\sqrt{2}\bigg(\frac{\Gamma(\frac{q+1}{2})}{\sqrt{\pi}}\bigg)^{1/q}, 
\end{equation}
where $\gamma$ is the one-dimensional standard Gaussian distribution and $\Gamma(x)=\int_0^{\infty}t^{x-1}e^{-t}\dd t$, $x>0$, denotes the Gamma function.

Together with the LDP in Theorem~\ref{thm:proj} we obtain the following corollary on the convergence of the random shadows $Z_{n,p}$, $n\ge k$, to the typical shadow $m_q \IB_2^k$.

\begin{cor}\label{cor:lln-proj}
	Let $2< p\le \infty$. The sequence $Z_{n,p}=n^{1/p-1/2}\Pi_{n,k}^{\top} \IB_p^n$, $n\ge k$, satisfies 
\begin{equation} \label{eq:lln-proj}
Z_{n,p}\xrightarrow[n\to\infty]{\rm a.s.} m_q \IB_2^k, 
\end{equation}
where almost sure convergence is in Hausdorff distance and $m_q$ is as in \eqref{eq:ball-min}.
\end{cor}

\begin{rem}\label{rem:dvoretzky}
	In the case of $p=\infty$, Corollary~\ref{cor:lln-proj} was obtained in \cite[Thm.~1.1]{KPT21}, where in Rem.~1.2 a proof using the Dvoretzky-Milman theorem in the form of \cite[Cor.~7.24]{AS17} is sketched. The proof of \cite[Rem.~1.2]{KPT21} can be extended to show that \eqref{eq:lln-proj} holds also for $1< p<\infty$. In this respect, note that the Dvoretzky-Milman theorem yields $c,c'>0$ such that for any convex body $K\subset\IR^n$ with nonempty interior, $\varepsilon>0$ and $k\le c\varepsilon^2 k_*$, it holds that, with probability at least $1-e^{-c'\varepsilon^2k_*}$, 
\[
(1-\varepsilon)w(K)\IB_2^k \subset \Pi_{n,k}^{\top}K \subset(1+\varepsilon)w(K)\IB_2^k,
\]
Here, $k_*$ is the critical/Dvoretzky dimension of the polar of $K$ and $w(K) =\int_{\Sp^{n-1}} h(K,u)\dd\sigma(u)$ is the half-mean width, where $\sigma$ is the rotation invariant probability measure on $\Sp^{n-1}$.  It is known (see e.g.~\cite[Sec.~7.2.4]{AS17}) that, for $g_1,\dots,g_n\overset{\rm i.i.d.}{\sim} \gamma$, 
\begin{equation} \label{eq:mean-width}
w(n^{1/p-1/2}\IB_p^n)
=\frac{\IE(\frac{1}{n}\sum_{i=1}^{n} |g_i|^q)^{1/q}}{\IE(\frac{1}{n}\sum_{i=1}^{n} |g_i|^2)^{1/2}}
\to \frac{m_q}{m_2}
=m_q,
\end{equation}
as $n\to\infty$ and that in this case $k_*$ grows like $n^{\min\{1,2/q\}}$. For $1<p<\infty$ and fixed $k\in\IN$, we can thus take $\varepsilon>0$ arbitrarily small as $n\to\infty$, and we obtain \eqref{eq:lln-proj} with a Borel-Cantelli argument. 
\end{rem}

By duality, we can obtain a result for the intersection of $\IB_p^n$ with a random $k$-dimensional subspace $E_{n,k}$, uniformly distributed on the Grassmannian $\mathbb{G}_{n,k}$ with respect to Haar measure. Let $\cK^k_{(o)}\subset \cK^k$ be the set of convex bodies containing the origin in their interior. For any $K\in \cK^k_{(o)}$, the polar or dual body is defined by
\[
K^{\circ}=\{x\in\IR^k\colon \langle x,y\rangle \le 1 \text{ for all }y\in K\}.
\]
By Proposition~\ref{pro:dual} below, it holds that 
\begin{equation} \label{eq:duality-distribution}
	K_{n,q}:=\Pi_{n,k}^{\top}(n^{1/q-1/2}\IB_q^n\cap E_{n,k})\overset{\rm d}{=} Z_{n,p}^{\circ},\qquad n\ge k,
\end{equation}
where $\Pi_{n,k}^{\top}\colon E_{n,k}\to \IR^k$ is an isometry and the polar $Z_{n,p}^{\circ}$ is well-defined almost surely. The following LDP is our main result for random sections.

\begin{thm}\label{thm:section}
	Let $1\le q<2$ and $2<p\le \infty$ with $\frac{1}{p}+\frac{1}{q}=1$. Fix $k\in \IN$. The sequence $K_{n,q}=\Pi_{n,k}^{\top}(n^{1/q-1/2}\IB_q^n\cap E_{n,k})$, $n\ge k$, satisfies an LDP in $\cK^k_{(o)}$	with good rate function
\[
J_q(K)=I_{p}(K^{\circ}),
\]
where $K\in \cK^{k}_{(o)}$ with polar $K^\circ\in \cK^k_{(o)}$.
\end{thm}

As a consequence of $(m_q\IB_2^k)^{\circ}=m_q^{-1}\IB_2^k$ we obtain the following corollary on the convergence of the random sections $K_{n,q}$, $n\ge k$, to the typical section $m_q^{-1}\IB_2^k$.

\begin{cor}\label{cor:lln-sec}
	Let $1\le q<2$. The sequence $K_{n,q}=\Pi_{n,k}^{\top}(n^{1/q-1/2}\IB_q^n\cap E_{n,k})$, $n\ge k$, satisfies
\[
K_{n,q}\xrightarrow[n\to\infty]{\rm a.s.} m_q^{-1} \IB_2^k, 
\]
where almost sure convergence is in Hausdorff distance and $m_q$ is as in \eqref{eq:ball-min}.
\end{cor}

Again, we may deduce this from the Dvoretzky-Milman theorem stated in its dual form for random sections. Note that a refinement for random sections of $\IB_p^n$ is given in \cite{PVZ17}. 

\begin{rem}
	We can infer from the LDPs in Theorems~\ref{thm:proj} and~\ref{thm:section} as well as the strong laws of large numbers in Corollaries~\ref{cor:lln-proj} and~\ref{cor:lln-sec} corresponding results for continuous functionals $L\colon \cK^k_{(o)}\to \IR$, via the contraction principle and the continuous mapping theorem, respectively. Examples include intrinsic volumes or diameter. Note that independently of us large deviations results for functionals of random projections and sections of $\IB_p^n$ are obtained in \cite{PTT24} using different techniques. Central limit theorems for the volume of random sections of fixed codimension have been established in \cite{APS24}. 
\end{rem}

\begin{rem}
Comparing the volume of random with extremal sections, it is due to Nazarov and proved in \cite{CNT22} that, for every $2$-dimensional linear subspace $H_{n,2}$ of $\IR^n$, 
\[
\vol_2(n^{1/2}\IB_1^n\cap H_{n,2})
\ge\frac{n^3\sin^3(\frac{\pi}{2n})}{\cos(\frac{\pi}{2n})},
\]
with equality when $\IB_{1}^n\cap H_{n,2}$ is a regular $2n$-gon. From above, we have the bound $\vol_2(\IB_1^n\cap H_{n,2})\le \vol_2(\IB_1^2)=2$, see \cite[Thm.~II.2]{MP88}. As $n\to\infty$, a regular $2n$-gon tends to an Euclidean ball and 
\[
\lim_{n\to\infty} \vol_2(n^{1/2}\IB_1^n\cap H_{n,2})
\ge \Big(\frac{\pi}{2}\Big)^3.
\]
It follows from Corollary~\ref{cor:lln-sec} and the continuous mapping theorem that 
\[
\vol_2(n^{1/2}\IB_1^n\cap E_{n,2})
\xrightarrow[n\to\infty]{\rm a.s.} \vol_2(m_1^{-1} \IB_2^2)
=\frac{\pi^2}{2},
\]
which makes the volume of random sections only slightly larger than for the minimal section. There is a vast literature on extremal sections and projections of convex bodies, in particular the $\ell_p$-balls $\IB_p^n$, see the survey \cite{NT23}.
\end{rem}

\begin{rem}\label{rem:one-dim}
In case of $k=1$ and $1<p\le\infty$, the random projection $Z_{n,p}$ is given by 
\begin{equation} \label{eq:one-dim}
	Z_{n,p}=[-W_{n,p},W_{n,p}] \quad \text{for } W_{n,p}=\frac{(\frac{1}{n}\sum_{i=1}^{n}|g_i|^q)^{1/q}}{(\frac{1}{n}\sum_{i=1}^{n}|g_i|^2)^{1/2}},
\end{equation}
where $g_1,\dots,g_n\overset{\rm i.i.d.}{\sim} \gamma$ and $\frac{1}{p}+\frac{1}{q}=1$, see Proposition~\ref{pro:proj-zonoid} below. In \cite[Sec.~5]{KPT19} it is proven that the sequence $(W_{n,p})_{n\in \IN}$ satisfies an LDP on $\IR$. Namely, if $p>2$ and thus $q<2$, then the rate function is given by virtue of Cram\'er's theorem in terms of the Legendre-Fenchel transform of the log-moment generating function $\Lambda(t_1,t_2)=\log \IE e^{t_1 |g_1|^q+t_2 |g_1|^2}, t_1,t_2\in \IR$. If $1<p<2$ and thus $2<q<\infty$, then $(W_{n,p})_{n\in \IN}$, satisfies an LDP with speed $n^{2/q}$ (replacing the prefactor $n$ in the above definition of LDP) and rate function 
\[
I_p^{(1)}(z)=
\begin{cases}
	\frac{1}{2}(z^q-m_q^q)^{2/q} &\colon z\ge m_q,\\
	\infty&\colon \text{otherwise,}
\end{cases}
\]
where $m_q$ is as in \eqref{eq:ball-min}. If $p=1$, then \eqref{eq:one-dim} holds with $W_{n,1}=\frac{\max_{1\le i\le n}|g_i|}{(\frac{1}{n}\sum_{i=1}^{n}|g_i|^2)^{1/2}}$. Similarly as in \cite{BKP+20} one can show that $(\frac{1}{\sqrt{2\log n}}W_{n,1})_{n\in\IN}$ satisfies an LDP with speed $\log n$ and rate function
\[
I_1^{(1)}(z)=
\begin{cases}
	z^2-1 &\colon z\ge 1,\\
	\infty&\colon \text{otherwise.}
\end{cases}
\]
Also in case of $k>1$, it seems possible to prove an LDP directly for the support function of $Z_{n,p}$, at least for $p>2$, however, with a less explicit rate function than in Theorem~\ref{thm:proj}. 
\end{rem}

\begin{rem}
	Our results for random projections and sections hold for $2< p\le \infty$ and $1\le q<2$, respectively. The case $p=q=2$ is trivial since $Z_{n,2}=\IB_2^k$, $n\ge k$. In the complementary ranges $1< p<2$ and $2\le q<\infty$, respectively, we expect different behavior in the form of large deviations with speed $n^{2/q}$. This is suggested by the Dvoretzky-Milman theorem (see Remark~\ref{rem:dvoretzky}), the one-dimensional case (see Remark~\ref{rem:one-dim}) or LDPs for functionals such as \cite[Thm.~1.2]{APT18}. The cases $p=1$ and $p=\infty$, respectively, closely resemble the setting of Gaussian polytopes but with only asymptotically independent points. This seems to require yet another approach. 
\end{rem}

As an application of the LDP in Theorem~\ref{thm:proj} we compute the asymptotic probability that the rescaled projection $Z_{n,p}$ of $\IB_p^n$ onto a random $k$-dimensional subspace lies within a Euclidean ball $\beta\IB_2^k$ of small radius $\beta>0$. By Corollary~\ref{cor:lln-proj} we have $Z_{n,p}\to m_q \IB_2^k$ and thus for $\beta<m_q$ the event $\{Z_{n,p}\subset \beta\IB_2^k\}$ is rare and occurs with exponentially small probability as specified by the following theorem. 

\begin{thm}\label{thm:small-ball}
Let $2< p\le \infty$ and $k\in \IN$. Let $1\le q<2$ with $\frac{1}{q}+\frac{1}{p}=1$ and
\[
\beta_{k,q}
=m_q\bigg(\frac{\frac{k}{q}\Gamma(\frac{k}{2})}{\Gamma(\frac{k+q}{2})}\bigg)^{1/q}
\bigg(\frac{\frac{k}{2}\Gamma(\frac{k}{q})}{\Gamma(\frac{k+2}{q})}\bigg)^{1/2}
=\delta_{k,q}^{1/q} \bigg(\frac{k\Gamma(\frac{k}{q})}{\Gamma(\frac{k+2}{q})}\bigg)^{1/2},
\]
where $m_q$ is as in \eqref{eq:ball-min} and $\delta_{k,q}$ is defined by the above equation. Then, it holds that $\beta_{k,q}< m_q$ and for every $\beta>0$, as $n\to\infty$, 
\[
\lim_{n\to\infty}n^{-1}\log\IP[Z_{n,p}\subset \beta\IB_2^k]
=c_{k,q,\beta},
\]
where for $\beta\ge m_q$, we have $c_{k,q,\beta}=0$, while for $\beta\le \beta_{k,q}$ we have
\[
c_{k,q,\beta}=k\log \beta-\frac{k}{2}\log(2\pi e)+\frac{k}{q}\Big(1-\log\delta_{k,q}\Big)+\log\Gamma\Big(\frac{k}{q}\Big)-\log q+\log\omega_k,
\]
where $\omega_k=\frac{2\pi^{k/2}}{\Gamma(\frac{k}{2})}$ is the surface area of $\IB_2^k$.
\end{thm}

In fact, the probability $\IP[Z_{n,p}\subset \beta\IB_2^k]$ is also exponentially vanishing for $\beta \in (\beta_{k,q},m_q)$ but the value of $c_{k,q,\beta}$ remains undetermined, see Section~\ref{sec:small-ball} for more information. A similar situation occurs in \cite[Sec.~4]{KR18} for the limit distribution of coordinates of a random vector uniformly distributed in an $\ell_p$-sphere and conditioned to lie in a small Euclidean ball. Note that $\beta_{k,q}\to 1=m_2$ as $q\to 2$, and so the gap $(\beta_{k,q},m_q)$ becomes smaller as $q\to 2$.

It is somewhat remarkable that, in contrast to previous large deviations results for random projections of $\ell_p$-balls in \cite{APT18,GKR17,KR23,LR24}, the explicit form of the rate function in Theorem~\ref{thm:proj} allows for a computation of the limit in Theorem~\ref{thm:small-ball} in terms of known functions. This suggests that it may be of advantage to study large deviations of geometric functionals, such as volume, of the random projections $Z_{n,p}$ from the perspective of $L_q$-zonoids and their generating measures. Note that in \cite{PP13} estimates on the small-ball probability for the volume of random zonoids were proven.

The proof of Theorem~\ref{thm:small-ball} relies on the representation of the rate function of Theorem~\ref{thm:proj} given in \eqref{eq:alternative-rate}, which involves computing a maximum entropy (probability) distribution. It is possible to compute the maximizer $\mu^*=f^*\dd x$ for special cases but seems difficult in general. However, it is conceivable that the maximizer takes the form
\[
f^*(x)=\frac{1}{Z}\exp\Big(-\int_{\Sp^{k-1}} |\langle x,u\rangle|^q \dd \nu_1(u)-\int_{\Sp^{k-1}} |\langle x,u\rangle|^2 \dd \nu_2(u)\Big),\qquad x\in \IR^k,
\]
where $\nu_1$ and $\nu_2$ are signed Borel measures on $\Sp^{k-1}$ and $Z$ is a normalization constant. For supporting arguments we refer to the proof of Theorem~\ref{thm:small-ball}. According to \cite[Thm.~3.1]{Kue73}, if $\nu_2=0$ and $\nu_1\ge 0$, then $f^*$ is the Fourier transform of a symmetric $q$-stable distribution. It would be interesting to obtain a geometric characterization of convex bodies $K$ having $f^* \dd x$ as maximum entropy distribution in the definition of $I_p(K)$. In this respect we note the following.

\begin{rem}\label{rem:q-projection}
An $L_q$-zonoid $K=Z_q(\mu)$ with $\mu\in\cP_q(\IR^k)$ may also be generated via
\begin{equation} \label{eq:spherical-gen}
h(K,u)
=\int_{\IR^k}|\langle x,u\rangle|^q \dd \mu(x)
=\int_{\Sp^{k-1}}|\langle z,u\rangle|^q\dd\nu(z),\qquad u\in \IR^k,
\end{equation}
by a measure $\nu$ assigning to a Borel set $A\subset \Sp^{k-1}$ the mass
\[
\nu(A)=\frac{1}{2}\int_{\IR^k} \|x\|_2^q \mathbf{1}_{\{x\neq 0\colon \frac{x}{\|x\|_2}\in A\cup -A\}}(x) \dd\mu(x).
\]
Koldobskii~\cite[p.~760]{Kol91} calls $\nu$ the $q$-projection of $\mu$ to $\Sp^{k-1}$, which is uniquely determined by $K$ via \eqref{eq:spherical-gen} if and only if $q\ge 1$ is not an even integer. This is well known and follows for example from \cite[Cor.~1]{Kan73}, \cite[Cor.~4]{Lin82} or \cite[Thm.~2]{Kol91}. See also \cite[Thm.~3.5.4]{Sch14} for $q=1$. According to \cite{Kan73}, uniqueness in the case $q<2$ goes back to L\'evy and relates to uniqueness of the spectral measure generating a symmetric $q$-stable distribution. Note that, if $k\ge 2$, $q>1$ is not an even integer and $\nu$ has a positive and bounded density with respect to the uniform measure $\sigma$, then $K$ given by \eqref{eq:spherical-gen} is smooth, i.e.~the boundary of $K$ is twice continuously differentiable and has everywhere positive Gauss-Kronecker curvature, see \cite[Thm.~2.2. and 3.1]{Lon03}.
\end{rem}

The following inequality on the Gamma function may be of independent interest.

\begin{prop}\label{pro:gamma}
It holds that, for any $x,y>0$,
\[
\bigg(\frac{x\Gamma(y)}{\Gamma(y+\frac{y}{x})}\bigg)^x
\bigg(\frac{y\Gamma(x)}{\Gamma(x+\frac{x}{y})}\bigg)^y\le 1,
\]
and equality holds if and only if $x=y$.
\end{prop}

With $x=\frac{k}{q}$ and $y=\frac{k}{2}$, Proposition~\ref{pro:gamma} implies $\beta_{k,q}< m_q$, which is claimed in Theorem~\ref{thm:small-ball}.

The structure of the remaining text is as follows. In Section~\ref{sec:zonoid} we provide some background on $L_q$-zonoids and in particular establish continuity of the map taking measures to the generated $L_q$-zonoids. Section~\ref{sec:proj} contains the proof of Theorem~\ref{thm:proj}. In Section~\ref{sec:sec} we use duality between sections and projections to deduce Theorem~\ref{thm:section}. Finally, in Section~\ref{sec:small-ball} we prove Theorem~\ref{thm:small-ball} and Proposition~\ref{pro:gamma}.

\section{$L_q$-zonoids}\label{sec:zonoid}

In the following, we collect facts about $L_q$-zonoids and refer to \cite[Sec. 10.14]{Sch14} and \cite[Sec.~10.4.2]{AGM15} for more information and references. In particular, we refer to \cite{Kol05} for an approach through Fourier analysis, to \cite{LYZ04} for volume estimates of $L_q$-zonoids and their polars and to \cite{Mol09} for a probabilistic interpretation of $L_q$-zonoids. If $q=1$, then $L_1$-zonoids are simply known as zonoids, see also \cite[Ch.~3.5]{Sch14}, and for $q=\infty$ one can define $Z_{\infty}(\mu)$ as the convex hull of the support of $\mu$. 

Let $1\le q<\infty$. Recall that the set of all $L_q$-zonoids is denoted by
\[
\Zon_{k,q}
=\{Z_q(\mu)\colon \mu\in\cP_q(\IR^k)\}
\subset \cK^k,
\]
and for $\mu\in \cP_q(\IR^k)$ the convex body $Z_q(\mu)$ has support function 
\begin{equation} \label{eq:tq}
T_q(\mu)\colon \IR^k\to \IR,\quad u\mapsto \Big(\int_{\IR^k}|\langle x,u\rangle|^q\dd\mu(x)\Big)^{1/q},
\end{equation}
whose $q$-th power is also called the $L_q$-cosine transform of $\mu$, see e.g.~\cite{Lon03}. It is well-defined since $T_q(\mu)(u)\le \big( \int_{\IR^k}\|x\|^q\dd\mu(x)\big)^{1/q}\|u\| <\infty$ for $u\in \IR^k$. In fact, allowing only measures in $\cP_q(\IR^k)$ is no restriction, which we show for completeness. 

\begin{lem}
Let $1\le q<\infty$ and $\mu\in \cP(\IR^k)$. If $\int_{\IR^k}|\langle x,u\rangle|^q\dd\mu(x)<\infty$ holds for every $u\in \IR^k$, then $\mu\in \cP_q(\IR^k)$.
\end{lem}
\begin{proof}
	Take a maximal set $\cE\subset \Sp^{k-1}$ such that $\langle u,v\rangle\le \frac{1}{2}$ for every $u,v\in \cE$ with $u\neq v$, i.e. a spherical code of minimal angle $\arccos(\frac{1}{2})$, see \cite[Ch.~1, 2.3]{CS99}. For any $x\in \IR^k\setminus\{0\}$ we find $v\in \cE$ such that $\langle\frac{x}{\|x\|},v\rangle> \frac{1}{2}$. Consequently, it holds that
\[
\frac{1}{2^q}\int_{\IR^k}\|x\|^q \dd\mu(x)
< \int_{\IR^k}\sup_{v\in \cE}|\langle x,v\rangle|^q \dd\mu(x)
\le \sum_{v\in \cE}\int_{\IR^k}|\langle x,v\rangle|^q \dd\mu(x)
<\infty,
\]
which gives $\mu\in \cP_q(\IR^k)$.
\end{proof}

In the following, we equip $\Zon_{k,q}$ with the induced Hausdorff metric and $\cP_q(\IR^k)$ with the $q$-Wasserstein metric, for which $\mu_n\to \mu$ in $\cP_q(\IR^k)$ if and only if $\mu_n\to \mu$ weakly and $\int_{\IR^k}\|x\|_2^q \dd\mu_n(x)\to\int_{\IR^k}\|x\|_2^q \dd\mu(x)$.

\begin{prop}\label{pro:continuity}
Let $1\le q<\infty$. The map
\[
Z_q\colon \cP_q(\IR^k)\to \Zon_{k,q},\qquad \mu \mapsto Z_q(\mu),
\]
is well-defined and continuous.
\end{prop}

For the proof of Proposition~\ref{pro:continuity} we need the following lemma, where we recall that a function $f\colon \IR^k\to\IR$ is sublinear if it is positively homogenous and subadditive.

\begin{lem}\label{lem:continuity}
Let $1\le q<\infty$. For every $\mu\in \cP_q(\IR^k)$ the function $T_q(\mu)\colon \IR^k\to \IR$ defined in \eqref{eq:tq} is Lipschitz continuous and sublinear.
\end{lem}
\begin{proof}
Let $\mu\in \cP_q(\IR^k)$. Positive homogenity of $T_q(\mu)$ is immediate. By the reverse triangle inequality in $L_q(\mu)$, for $u_1,u_2\in \IR^k$,
\[
|T_q(\mu)(u_1)-T_q(\mu)(u_2)|
\le T_q(\mu)(u_1-u_2)
\le \Big(\int_{\IR^k}\|x\|^q\dd\mu(x)\Big)^{1/q}\|u_1-u_2\|,
\]
which establishes subadditivity and Lipschitz continuity of $T_q(\mu)$.  
\end{proof}

\begin{proof}[Proof of Proposition~\ref{pro:continuity}]
	Denote by $C(\Sp^{k-1})$ the space of continuous functions on $\Sp^{k-1}$ equipped with the uniform metric and by $C_s(\Sp^{k-1})$ the set of $f\in C(\Sp^{k-1})$ for which the extension $\tilde{f}(x)=\|x\| f(\frac{x}{\|x\|})$, $x\in \IR^k\setminus\{0\}$, $\tilde{f}(0)=0$, is sublinear.	By Lemma~\ref{lem:continuity}, for every $\mu\in \cP_q(\IR^k)$ the restriction $T_q(\mu)|_{\Sp^{k-1}}$ belongs to $C_s(\Sp^{k-1})$.  Thus, the map $\tilde{T}_q\colon \cP_q(\IR^k)\to C_s(\Sp^{k-1})$ taking $\mu$ to $T_q(\mu)|_{\Sp^{k-1}}$ is well-defined. 

We can then write $Z_q=I^{-1}\circ \tilde{T}_q$, where
\begin{align*}
I\colon \cK^k \to C_s(\Sp^{k-1}),\qquad K \mapsto h(K,\cdot)|_{\Sp^{k-1}},
\end{align*}
is an isometric bijection, see e.g. \cite[Ch.~1]{Sch14}. Therefore, the composition $Z_q=I^{-1}\circ \tilde{T}_q$ is well-defined and it only remains to prove continuity of $\tilde{T}_q\colon \cP_q(\IR^k)\to C_s(\Sp^{k-1})$. 

Let $\mu_0,\mu_1,\dots$ with $\mu_n\to \mu_0$ in $\cP_q(\IR^k)$, i.e.~$\mu_n\to\mu_0$ weakly and moreover the convergence $\int_{\IR^k}\|x\|^q\dd\mu_n(x)\to\int_{\IR^k}\|x\|^q\dd\mu_0(x)$ holds. Moreover, for every $u\in\Sp^{k-1}$, the pointwise convergence $\tilde{T}_q(\mu_n)(u)\to \tilde{T}_q(\mu_0)(u)$ holds, see e.g.~\cite[Ch.~6.8]{Vil09}. By Lemma~\ref{lem:continuity}, the functions $\tilde{T}_q(\mu_0),\tilde{T}_q(\mu_1),\dots$ are Lipschitz-continuous on $\Sp^{k-1}$ (with respect to the Euclidean metric) with Lipschitz-constant at most $M^{1/q}$ where $M:=\sup_{n\in\IN_0}\int_{\IR^k}\|x\|^q\dd\mu_n(x)$ is finite. By the Arzel\`a-Ascoli theorem, the sequence $(\tilde{T}_q(\mu_n))_{n\in\IN}$ is relatively compact in $C(\Sp^{k-1})$ and uniform convergence $\tilde{T}_q(\mu_n)\to \tilde{T}_q(\mu_0)$ follows from pointwise convergence. This proves Proposition~\ref{pro:continuity}.
\end{proof}

We note the following consequence of Lemma~\ref{lem:continuity}. Recall that for our purposes a convex body may have empty interior.

\begin{prop}\label{pro:interior}
For a measure $\mu\in \cP_q(\IR^k)$ the following statements are equivalent:
\begin{enumerate}
	\item[(i)] $\mu$ is not supported on a hyperplane,
	\item[(ii)] $Z_q(\mu)$ contains the origin in its interior,
	\item[(iii)] $Z_q(\mu)$ is a symmetric convex body with nonempty interior.
\end{enumerate}
\end{prop}
\begin{proof}
(i) $\Leftrightarrow $ (ii)$\colon $  Suppose that $\mu\in \cP_q(\IR^k)$ is supported on a hyperplane $H_u=\{x\in\IR^k\colon \langle x,u\rangle = 0\}$. By definition, the support function satisfies then $h(Z_q(\mu),\pm u)=0$, implying that $Z_q(\mu)\subset H_u$. Consequently, the interior of $Z_q(\mu)$ is empty.

If, on the other hand, $Z_q(\mu)$ does not contain the origin in its interior, then for arbitrary $\varepsilon>0$ there is $u\in \Sp^{k-1}$ with $h(Z_q(\mu),u)<\varepsilon$. This gives $\inf_{u\in\Sp^{k-1}}h(Z_q(\mu),u)=0$. By Lemma~\ref{lem:continuity}, the function $h(Z_q(\mu),\cdot)$ is continuous and attains its infimum at some point $u_0\in \Sp^{k-1}$. Then $\mu$ is supported in $H_{u_0}$. This completes the proof of the first equivalence.

(ii) $\Leftrightarrow $ (iii)$\colon $ By Lemma~\ref{lem:continuity}, the function $T_q(\mu)$ is sublinear and thus the support function of a convex body given by $Z_q(\mu)$, which is symmetric. By symmetry, the $L_q$-zonoid $Z_q(\mu)$ has nonempty interior if and only if $0\in \Int Z_q(\mu)$.
\end{proof}

The set $\Zon_{k,q}$ of $L_q$-zonoids arises from the closure of the set of $L_q$-zonotopes, where $K\in \cK^k$ is called an $L_q$-zonotope if its support function is of the form
\begin{equation} \label{eq:q-zonotope}
h(K,u)=\Big(\sum_{i=1}^{n}|\langle a_i,u\rangle|^q\Big)^{1/q},\qquad u\in \IR^k,
\end{equation}
for some $a_1,\dots,a_n\in \IR^k$. This is a special case of \eqref{eq:q-zonoid} with discrete generating measure $\mu$. We can also interpret $K$ as the $q$-sum of the line segments $[-a_i,a_i]$, $i=1,\dots,n$, where the $q$-sum of convex bodies $K,L$ containing the origin is the unique convex body with support function $u\mapsto (h(K,u)^q+h(L,u)^q)^{1/q}$, $u\in \IR^k$, see \cite{Fir62} and \cite[Ch.~9.1]{Sch14}. We observe that $L_q$-zonotopes correspond to linear images of $\IB_p^n$.

\begin{lem}\label{lem:zonotope}
Let $1< p\le \infty$ and $1\le q<\infty$ with $\frac{1}{q}+\frac{1}{p}=1$ and $n,k\in \IN$. For every $A\in \IR^{k\times n}$ the image $A\IB_p^n$ is a convex body in $\IR^k$ with support function of the form \eqref{eq:q-zonotope}, where $a_1,\dots,a_n$ are the columns of $A$. In particular, it is an $L_q$-zonoid  with generating measure $\frac{1}{n}\sum_{i=1}^{n} \delta_{n^{1/q}a_i}$.
\end{lem}
\begin{proof}
The support function of $A\IB_p^n$ at $u\in\IR^k$ is given by	
\[
h(A\IB_p^n,u)
=\sup_{x\in \IB_{p}^n}\langle x,A^{\top} u\rangle
=\|A^{\top} u\|_q.
\]
Denoting the columns of $A$ by $a_1,\dots,a_n$, it follows for $ \mu_A =\frac{1}{n}\sum_{i=1}^{n}\delta_{n^{1/q}a_i} $ that
\[
h(A\IB_p^n,u)
=\Big(\sum_{i=1}^{n}|\langle a_i,u\rangle|^q\Big)^{1/q}
=\Big(\int_{\IR^k}|\langle x,u\rangle |^q {\rm d} \mu_A(x)\Big)^{1/q}.
\]
\end{proof}

\begin{rem}
Lemma~\ref{lem:zonotope} admits a partial converse. Let $1\le p\le \infty$ with $p\neq 2$ and $1\le q<\infty$. If for some $A\in \IR^{k\times n}$ the set $A\IB_p^n$ is not contained in a two-dimensional subspace, then it is an $L_q$-zonoid if and only if $p^*=q$ or $1\le q\le  p^* \le 2$, where $\frac{1}{p}+\frac{1}{p^*}=1$. In particular, for $p<2$, the set $A\IB_p^n$ is an $L_q$-zonoid if and only if $p^*=q$. This follows from the embedding results in \cite[Thm.~2.1]{Dor76} together with \cite[Lem.~6.4]{Kol05} and Remark~\ref{rem:q-projection}. 
\end{rem}

\section{Large deviations in the Stiefel manifold and the proof of Theorem~\ref{thm:proj}}\label{sec:proj}

In the following, we derive Theorem~\ref{thm:proj} from an LDP for the empirical measure of rescaled rows of the random Stiefel matrix $\Pi_{n,k}$. The following is a special case of Lemma~\ref{lem:zonotope}.

\begin{prop}\label{pro:proj-zonoid}
	Let $1\le p\le \infty$ and $n\ge k$. Every realization of $ Z_{n,p}=n^{1/p-1/2}\Pi_{n,k}^{\top}\mathbb{B}_p^n $ is an $L_q$-zonoid with generating measure 
\begin{equation} \label{eq:empirical}
L_n =\frac{1}{n}\sum_{i=1}^{n}\delta_{\sqrt{n}v_i},
\end{equation}
where $v_1,\dots,v_n$ are the rows of $\Pi_{n,k}$. 
\end{prop}

\begin{rem}\label{rem:q-sum}
Proposition~\ref{pro:proj-zonoid} can be interpreted as $Z_{n,p}$ being the $q$-sum of the random line segments $[-\sqrt{n}v_i,\sqrt{n}v_i]$, $i=1,\dots,n$. Since these segments are identically distributed, one approach to an LDP for $(Z_{n,p})_{n\ge k}$ may consist in extending existing LDPs for Minkowski sums to $q$-sums. However, the dependence between the rows $v_1,\dots,v_n$ complicates such an approach. 
\end{rem}

Kim and Ramanan established in \cite{KR23} that for every $q\in (0,2)$ the sequence $(L_n)_{n\ge k}$ in \eqref{eq:empirical} satisfies an LDP in the $q$-Wasserstein space $\cP_q(\IR^k)$ equipped with the $q$-Wasserstein metric. Note that an LDP directly for the sequence $(\Pi_{n,k})_{n\ge k}$ was proven in \cite{KP24}. It is well-known that $\Pi_{n,k}\overset{\rm d}{=} G_{n,k}^{\top}(G_{n,k}G_{n,k}^{\top})^{-1/2}$, where $G_{n,k}$ is a Gaussian random matrix with columns $g_1,\dots,g_n\overset{\rm i.i.d.}{\sim} \gamma^{\otimes k}$, see e.g.~\cite[Lem.~3.1]{KPT21}. By the law of large numbers, the random matrix $\frac{1}{n}G_{n,k}G_{n,k}^*$ tends almost surely to the identity matrix $I_{k\times k}$, as $n\to\infty$, and thus one may compare the behavior of $L_n$ as in \eqref{eq:empirical} to that of $\frac{1}{n}\sum_{i=1}^{n}\delta_{g_i}$. By Sanov's theorem, the latter sequence of empirical measures satisfies an LDP in $\cP(\IR^k)$, equipped with the weak topology, with rate function at $\mu\in \cP(\IR^k)$ given by the relative entropy
\[
H(\mu | \gamma^{\otimes k})
=
\begin{cases}
	\int_{\IR^k} \log(\frac{\dd \mu}{\dd \gamma^{\otimes k}})\dd\mu &  \colon   \mu\ll \gamma^{\otimes k},\\
	\infty & \colon  \text{otherwise}.
\end{cases}
\]
In contrast, the rate function corresponding to $L_n$ is given by
\begin{equation} \label{eq:stiefel-grf}
\IH_k(\mu)
=
\begin{cases}
	H(\mu | \gamma^{\otimes k})+\frac{1}{2}\tr(I_{k\times k}-\Cov(\mu)) & \colon  \Cov(\mu)\le I_{k\times k},\\
	\infty & \colon  \text{otherwise},
\end{cases}
\end{equation}
where $\tr(A)$ denotes the trace of a matrix $A\in \IR^{k\times k}$.

\begin{prop}[{\cite[Thm. 2.8]{KR23}}]\label{pro:measures}
	Let $q\in (0,2)$ and $k\in\IN$. Moreover, let $v_1,\dots,v_n$ be the rows of a random matrix $\Pi_{n,k}$ sampled according to the Haar measure on $\mathbb{V}_{n,k}$. Then, the sequence $(L_n)_{n\ge k}$ as in \eqref{eq:empirical} satisfies an LDP in $\cP_q(\IR^k)$ (equipped with the $q$-Wasserstein metric) with strictly convex good rate function $\IH_k$.
\end{prop}

We shall infer the LDP in Theorem~\ref{thm:proj} from Proposition~\ref{pro:measures} using continuity of the map $Z_q$ taking the random measure $L_n$ to the random convex body $Z_q(L_n)=Z_{n,p}$. 

\begin{proof}[Proof of Theorem~\ref{thm:proj}]
Let $2< p\le \infty$ and $1\le q<2$ with $\frac{1}{p}+\frac{1}{q}=1$. Using the contraction principle (e.g. \cite[Thm.~4.2.1]{DZ10}), we derive from Propositions~\ref{pro:continuity} and~\ref{pro:measures} that the sequence $(Z_{n,p})_{n\ge k}$ satisfies a large deviation principle in $\Zon_{k,q}$ with rate function
\begin{equation} \label{eq:rate-function}
I_p(K)
=\inf\{\IH_k(\mu)\colon K=Z_q(\mu)\}, \qquad K\in \Zon_{k,q}.
\end{equation}
Since $\Zon_{k,q}$ is closed and $I_p(K)=\infty$ for $K\in\cK^k\setminus \Zon_{k,q}$, the LDP extends to $\cK^k$, see \cite[Lem.~4.1.5]{DZ10}.

In the following, we rewrite the rate function in \eqref{eq:rate-function} to the one in Theorem~\ref{thm:proj}. Plugging in the definition of $\IH_k$ from \eqref{eq:stiefel-grf}, we have $I_p(K)<\infty$ precisely when $K=Z_q(\mu)$ for a probability measure $\mu\in \cP_q(\IR^k)$ which admits a Lebesgue density $f$ and satisfies $\Cov(\mu)\le I_{k\times k}$ as well as $H(\mu|\gamma^{\otimes k})<\infty$. In this case, we compute 
\[
\IH_k(\mu)
=-\Ent(\mu)+\frac{1}{2}\int_{\IR^k}\|x\|^2\dd\mu(x)+\frac{k}{2}(\log(2\pi)+1)-\frac{1}{2}\tr(\Cov(\mu)).
\]
Further, we note that
\[
\tr(\Cov(\mu))
=\int_{\IR^k}\|x\|^2\dd\mu(x)-\sum_{i=1}^{n}\Big(\int_{\IR^k}x_i\dd\mu(x)\Big)^2
=\int_{\IR^k}\|x\|^2\dd\mu(x)-\|\bc(\mu)\|^2,
\]
where $\bc(\mu)=\int_{\IR^k} x \, \dd\mu(x)$ is the barycenter of $\mu$. Together with $\Ent(\gamma^{\otimes k})=\frac{k}{2}\log(2\pi e)$ we arrive at
\[
\IH_k(\mu)
={\rm Ent}(\gamma^{\otimes k})-{\rm Ent}(\mu)+\frac{1}{2}\|\bc(\mu)\|^2.
\]
The symmetrized version $\tilde{\mu}$ of $\mu$ has Lebesgue density $\tilde{f}(x)=\frac{f(x)+f(-x)}{2}$, satisfies $K=Z_q(\tilde{\mu})$ and is centered, i.e. $\bc(\tilde{\mu})=0$. Further, by convexity of $x\mapsto x\log x$, it holds that
$
-{\rm Ent}(\tilde{\mu})
\le -{\rm Ent}(\mu),
$
and consequently $\IH_k(\tilde{\mu})\le \IH_k(\mu)$. Thus, we can put 
\[
I_p(K)=\inf\{\Ent(\gamma^{\otimes k})-{\rm Ent}(\mu)\colon \mu\in\cP_q(\IR^k) \text{ with } Z_q(\mu)=K \text{ and } \Cov(\mu)\le I_{k\times k}\}.
\]
Finally, rewrite $\Cov(\mu)\le I_{k\times k}$ to $Z_2(\mu)\subset \IB_2^k$. This completes the proof of Theorem~\ref{thm:proj}. 
\end{proof}

\section{Duality and the proof of Theorem~\ref{thm:section}}\label{sec:sec}

We use basic facts about duality and functions representing convex sets, see also \cite{Sch14}. Let $K\subset \R^k$ be a convex body such that $0\in \Int K$, that is $K\in \cK^{k}_{(o)}$. Then, its dual $K^\circ$ is also a convex body with $0\in \Int K^\circ$.  Recall that the support function of $K$ is given by
\[
h(K, u) = \sup_{x\in K} \langle x, u\rangle, \qquad u\in \bS^{k-1},
\]
and the radial function of $K$ is defined by
\[
\rho(K, u) = \sup\{r\ge 0: ru \in K\}, \qquad u\in \bS^{k-1}.
\]
It is known that there is a duality between the support function and the reciprocal radial function, namely
\[
h(K^\circ, u) = 1/ \rho(K, u), \qquad u\in \bS^{k-1}.
\]
It is clear that the radial function determines $K$ uniquely.  As a consequence of the above duality, the support function also determines $K$ uniquely. We have the following duality between sections and projections, see also \cite[Sec.~1.1.4]{AS17}.

\begin{proposition}\label{pro:dual}
Let $K\subset \R^k$ be a convex body such that $0\in \Int K$. Let also $L\subset \R^k$ be a linear subspace and $P_L$ be the orthogonal projection onto $L$. Then, the convex bodies $K^\circ\cap L$ and $P_L K$, considered within the ambient space $L$, are dual to each other.
\end{proposition}
\begin{proof}
If we want to indicate that the support function is considered with respect to the ambient space $L$, we will denote it by $h_L$.
For $u\in \bS^{k-1}\cap L$ we have
\begin{align*}
h_L(P_L K,u)
=
\sup_{x\in P_L K} \langle x, u\rangle
=
\sup_{x\in K} \langle x, u\rangle
=
h_{\R^k} (K, u).
\end{align*}
It follows that
\[
h_L(P_L K,u)
=
h_{\R^k} (K, u)
=
1/\rho(K^\circ, u)
=
1/\rho(K^\circ\cap L, u)
=
h_L((K^\circ\cap L)^\circ_L, u),
\]
where in the last term the subscript $L$ indicates that the dual is taken within $L$.
Since the support function determines the convex body uniquely, we conclude that $P_L K$ and $K^\circ\cap L$ are dual to each other within $L$.
\end{proof}

We recall that since the random $k$-dimensional subspace $E_{n,k}$ is uniformly distributed on the Grassmannian $\mathbb{G}_{n,k}$, we can put $E_{n,k}= \Pi_{n,k}(\IR^k)$ and then $\Pi_{n,k}^{\top}E_{n,k}= \IR^k$, where $\Pi_{n,k}$ is uniformly distributed on $\mathbb{V}_{n,k}$. Almost surely, $\Pi_{n,k}$ is of full rank, which we shall assume in the following.

Let $1\le q<2$. Proposition~\ref{pro:dual} and the duality between $\IB_{p}^n$ and $\IB_{q}^n$ imply
\[
(n^{1/q-1/2}\IB_q^n\cap E_{n,k})^{\circ}
=n^{1/p-1/2} P_{E_{n,k}}\IB_{p}^n,
\]
where the dual is taken with respect to the ambient space $E_{n,k}$.  Since $(\Pi_{n,k}(\IR^k))^{\perp}=\ker\Pi_{n,k}^{\top}$, it holds that $\Pi_{n,k}^{\top} P_{E_{n,k}}=\Pi_{n,k}^{\top}$, and we can apply the isometry $\Pi_{n,k}^{\top}\colon E_{n,k}\to \IR^k$, which commutes with duality, to both sides and arrive at
\[
(\Pi_{n,k}^{\top}(n^{1/q-1/2}\IB_q^n\cap E_{n,k}))^{\circ}
= n^{1/p-1/2} \Pi_{n,k}^{\top}\IB_{p}^n
=Z_{n,p}.
\]
Polarity being an involution, we obtain \eqref{eq:duality-distribution}, i.e., that $K_{n,q}=Z_{n,p}^{\circ}$, where $K_{n,q}$ is as in Theorem~\ref{thm:section}. 

In the following, we establish continuity of the polar operation, and for this we need that convergence in Hausdorff distance implies uniform convergence of radial functions.  

\begin{lem}\label{lem:hd-rad}
	Let $K,K_1,K_2,\dots\in \cK^k_{(o)}$. If $K_n\to K$, then the uniform convergence $\rho(K_n,\cdot)\to \rho(K,\cdot)$ holds on $\Sp^{k-1}$.
\end{lem}
\begin{proof}
Let $0<r<R<\infty$ be such that $B(0,2r)\subset K \subset B(0,R/2)$. Then for $n\ge n_0$ we have $B(0,r)\subset K_n \subset B(0,R)$. Let $\varepsilon>0$ be small and $\delta=\delta(r,R,\varepsilon)$ be such that for any $\alpha\in (\varepsilon,1)$ and $x\in B(0,R)$ we have
\[
B\Big(x-\alpha\frac{x}{\|x\|},2\delta\Big)\subset C_x:={\rm conv}(\{x\}\cup B(0,r)),
\]
see e.g. \cite[Ch.~3.2]{Wen04}. Intuitively, the cone $C_x$ is not too narrow away from its apex. Assume that $|\rho(K_n,u_n)-\rho(K,u_n)|>\varepsilon$ for some $u_n\in \Sp^{k-1}$ and $n\ge n_0$. It remains to show $\delta_H(K_n,K)\ge\delta$.

In fact, suppose without loss of generality that $r\le \rho(K_n,u_n)<\rho(K,u_n)\le R$. Further, $x_n:=\rho(K,u_n)u_n\in \partial K$ and $C_{x_n}\subset K$ as well as $y_n:=\rho(K_n,u_n)u_n\in \partial K_n$. Moreover, $y_n=x_n-\alpha_n\frac{x_n}{\|x_n\|}$ for some $\alpha_n=\rho(K,u_n)-\rho(K_n,u_n)>\varepsilon$. The ball $B(y_n,\delta)$ is contained in the interior of $C_{x_n}$ and thus of $K$. Therefore, $\delta_H(K_n,K)=\delta_H(\partial K_n, \partial K)\ge \delta$.
\end{proof}

\begin{lem}
	Polarity $\cdot^{\circ}\colon \cK^{k}_{(o)}\to \cK^k_{(o)}$ is continuous with respect to Hausdorff distance.	
\end{lem}
\begin{proof}
	Assume $K,K_1,K_2,\dots$ are in $\cK^k_{(o)}$ and that $K_n\to K$ in Hausdorff distance. By Lemma~\ref{lem:hd-rad} we have the uniform convergence $\rho(K_n,\cdot)\to \rho(K,\cdot)$. Since $K\in \cK^k_{(o)}$, we find $r>0$ such that for large enough $n$ and all $u\in \Sp^{k-1}$, we have $\rho(K_n,u)\ge r$. Thus, we have uniform convergence of 
\[
h(K_n^{\circ},\cdot)
=\rho(K_n,\cdot)^{-1}
\to \rho(K,\cdot)^{-1}
= h(K^{\circ},\cdot),
\]
showing that $K_n^{\circ}\to K^{\circ}$ in Hausdorff distance.
\end{proof}

We are now ready to prove Theorem~\ref{thm:section}. 

\begin{proof}[Proof of Theorem~\ref{thm:section}]
	Theorem~\ref{thm:proj} implies that for $2<p\le \infty$ the sequence $(Z_{n,p})_{n\ge k}$ satisfies an LDP in $\cK^k$ with good rate function $I_p$ as defined in Theorem~\ref{thm:proj}. As noted in Remark~\ref{rem:finite-rate}, Proposition~\ref{pro:interior} implies that $\{K\in\cK^k\colon I_p(K)<\infty\}\subset \cK^k_{(o)}$. Further, $\cK^k_{(o)}\subset \cK^k$ is measurable by way of $\cK^{k}_{(o)}=\bigcup_{n=1}^{\infty}\{K\colon B(0,\frac{1}{n})\subset K\}$, where the sets in the union are open (see the proof of \cite[Lem.~1.8.18]{Sch14}). Thus, by \cite[Lem.~4.1.5]{DZ10} the sequence $(Z_{n,p})_{n\ge k}$ satisfies an LDP in $\cK^k_{(o)}$ with the same good rate function $I_p$.

Let $1\le q<2$ with $\frac{1}{p}+\frac{1}{q}=1$. Then by \eqref{eq:duality-distribution} it holds that $K_{n,q}=Z_{n,p}^{\circ}$. The contraction principle implies that the sequence $(K_{n,q})_{n\ge k}$ satisfies an LDP on $\cK^{k}_{(o)}$ with good rate function $J_q\colon \cK^{k}_{(o)}\to [0,\infty]$ defined by
\[
J_q(K)
=\inf\{I_p(L)\colon L\in \cK^k_{(o)} \text{ with }L^{\circ}=K\}
=I_p(K^{\circ}).
\]
This completes the proof of Theorem~\ref{thm:section}.
\end{proof}

\section{Maximum entropy and the proof of Theorem~\ref{thm:small-ball}}\label{sec:small-ball}

Let $2<p\le \infty$ and thus $1\le q<2$ with $\frac{1}{p}+\frac{1}{q}=1$. For $\beta> 0$ define the set 
\begin{equation} \label{eq:A-beta}
A_{\beta}=\{K\in \cK^k\colon K\subseteq \beta \IB_2^k\}.
\end{equation}
Observe that $A_{\beta}$ is closed with interior $\Int A_{\beta}=\{K\in \cK^k\colon K\subset  \beta \Int\IB_2^k\}$. Once we can show that, for $I_p$ as in Theorem~\ref{thm:proj} and $c_{k,q,\beta}$ as in Theorem~\ref{thm:small-ball}, 
\begin{equation} \label{eq:continuity-set}
\inf_{K\in \Int A_{\beta}}I_p(K)
=\inf_{K\in A_{\beta}}I_p(K)
=-c_{k,q,\beta},
\end{equation}
then the limit in Theorem~\ref{thm:small-ball} follows from the large deviation principle in Theorem~\ref{thm:proj} applied to the Borel set $A_{\beta}\subset \cK^k$. In the course of the proof of \eqref{eq:continuity-set}, we will give the proof of Proposition~\ref{pro:gamma}, which implies the inequality $\beta_{k,q}<m_q$ claimed in Theorem~\ref{thm:small-ball}. 

By the alternative representation \eqref{eq:alternative-rate} of the rate function $I_p$ we have 
\begin{equation} \label{eq:inf-ent-1}
\inf_{K\in A_{\beta}}I_p(K)
=\frac{k}{2}\log(2\pi e)-\sup_{\mu\in M_{\beta}}\Ent(\mu),
\end{equation}
where 
\begin{equation} \label{eq:M-beta}
M_{\beta}=\{\mu\in\cP_q(\IR^k)\colon  Z_q(\mu)\subset \beta \IB_2^k \text{ and } Z_2(\mu)\subset \IB_2^k\},
\end{equation}
which may also be written as the set of probability measures $\mu\in\cP_q(\IR^k)$ satisfying
\begin{align}\label{eq:constraints}
	\int_{\IR^k}|\langle x,u\rangle|^q \dd\mu(x)&\le \beta^q, \qquad \text{for all }u\in\Sp^{k-1},\notag\\
	\int_{\IR^k}|\langle x,u\rangle|^2 \dd\mu(x)&\le 1, \qquad \text{for all }u\in\Sp^{k-1}.
\end{align}
Analogously, we have
\begin{equation} \label{eq:inf-ent-2}
\inf_{K\in \Int A_{\beta}}I_p(K)
=\frac{k}{2}\log(2\pi e)-\sup_{\mu\in M_{\beta}'}\Ent(\mu),
\end{equation}
with
\[
M_{\beta}'
=\{\mu\in\cP_q(\IR^k)\colon  Z_q(\mu)\subset \beta \Int\IB_2^k \text{ and } Z_2(\mu)\subset \IB_2^k\},
\]
which can be written in terms of the constraints \eqref{eq:constraints} with the first inequality replaced by a strict inequality.

The proof strategy for showing \eqref{eq:continuity-set} is to compute the supremum $\sup_{\mu \in M_{\beta}}\Ent(\mu)$ and to show that this value does not change if $M_{\beta}$ is replaced by $M_{\beta}'$. The following proposition shows that the supremum of $\Ent(\mu)$ over $\mu\in M_{\beta}$ is actually attained and the maximizer is unique if the supremum is finite. In the latter case we call the maximizer the maximum entropy distribution in $M_{\beta}$. 

\begin{prop}\label{pro:attained}
	For any $\beta>0$ the set $M_{\beta}$ is compact in $q$-Wasserstein distance and the supremum $\sup_{\mu\in M_{\beta}}\Ent(\mu)$ is attained at some $\mu^*$. If the supremum is finite, then the maximizer is unique.
\end{prop}
\begin{proof}
For compactness, we observe that $M_{\beta}$ is the intersection of
\[
C
=\{\mu\in \cP_q(\IR^k)\colon Z_2(\mu)\subset \IB_2^k\}
=\{\mu\in \cP_q(\IR^k)\colon \Cov(\mu)\le I_{k\times k}\}
\]
with the preimage of the closed set $A_{\beta}$ under the continuous map $Z_q\colon \cP_{q}(\IR^k)\to \cK^k$ (see Proposition~\ref{pro:continuity}). Thus, it is sufficient to establish compactness of $C$. This is a straightforward generalization of \cite[Lem.~3]{KR18} (recall that $q<2$). In particular, for tightness observe that any $\mu\in C$ satisfies
\[
\int_{\IR^k} \|x\|^2 \dd\mu(x)
=\tr(\Cov(\mu))
\le k.
\]
Consequently, the set $M_{\beta}$ is compact. Considering entropy as a function
\[
\Ent\colon \cP_q(\IR^k)\to \IR\cup \{-\infty\},
\]
it is upper semi-continuous with respect to $q$-Wasserstein topology, see e.g.~\cite[Rem.~9.3.8]{AGS05}, and therefore attains its maximum on the compact convex set $M_{\beta}$. 

Let $\mu_1^*,\mu_2^*\in M_{\beta}$ be two maximizers of the entropy with $\Ent(\mu_1^*)=\Ent(\mu_2^*)<\infty$. Then $\mu_1^*$ and $\mu_2^*$ have Lebesgue densities $f_1^*$ and $f_2^*$. Consider a measure $\mu_3^*\in M_{\beta}$ with density $\frac{1}{2}(f_1^*+f_2^*)$. By strict concavity of $x\mapsto -x\log x$ we have $\Ent(\mu_3^*)>\Ent(\mu_1^*)$, which is a contradiction. Hence, the maximizer is unique.
\end{proof}

In the following, we compute the maximum entropy distribution $\mu^*$ in $M_{\beta}$ for $\beta>0$ and the corresponding entropy, which will be shown to be finite. Using rotational symmetry of $M_{\beta}$, the computation of $\sup_{\mu\in M_{\beta}}\Ent(\mu)$ will reduce to a one-dimensional problem with an explicit solution if $\beta\le \beta_{k,q}$ or $\beta\ge m_q$.

More precisely, for any orthogonal matrix $O$ in the orthogonal group $\mathbb{O}_k$, we have $\mu\in M_{\beta}$ if and only if $O\cdot \mu\in M_{\beta}$, where the rotated measure $O\cdot\mu$ assigns to a Borel set $A\subset \IR^k$ the mass
\[
(O\cdot \mu)(A)
:=\mu(\{O^{-1}x\colon x\in A\}).
\]
Moreover, the rotation invariance of $M_{\beta}$ transfers to the rotation invariance of the maximum entropy distribution $\mu^*$.

\begin{lem}\label{lem:max-invariant}
If a compact convex set $M\subset \cP_q(\IR^k)$ is invariant under rotations, then so is the maximum entropy distribution in $M$, provided that it has finite entropy.
\end{lem}
\begin{proof}
	Since $M$ is compact and convex, the maximizer of the entropy is unique, see the proof of Proposition~\ref{pro:attained}. If $\mu^*$ is a maximizer, then for every $O\in \mathbb{O}_k$ the measure $O\cdot \mu^*$ is also a maximizer, proving that $O\cdot \mu^*=\mu^*$.  
\end{proof}

By Lemma~\ref{lem:max-invariant}, provided that its entropy is finite, the maximum entropy distribution $\mu^*\in M_{\beta}$ is rotationally invariant, i.e. has a Lebesgue density $f\colon \IR^k\to \IR$ of the form $f\colon x\mapsto g(\|x\|)$ for $g\colon [0,\infty)\to [0,\infty)$ with
\[
1
=\int_{\IR^k}f(x)\dd x
=\omega_k\int_0^{\infty} r^{k-1}g(r)\dd r
=\int_0^{\infty}v(r)\dd r,
\]
where $v(r)=\omega_k r^{k-1}g(r)$. Again using polar integration, for any $q\ge 1$ and $u\in\Sp^{k-1}$,
\begin{equation} \label{eq:radial}
\int_{\IR^k}|\langle x,u\rangle|^q \dd\mu^*(x)
= a_{k,q}\int_{0}^{\infty}r^{q}v(r)\dd r,
\end{equation}
where
\begin{equation} \label{eq:akq}
	a_{k,q}
	:=\int_{\Sp^{k-1}}|\langle z,u\rangle|^q\dd\sigma(z)
=\frac{1}{\sqrt{\pi}}\frac{\Gamma(\frac{q+1}{2})\Gamma(\frac{k}{2})}{\Gamma(\frac{k+q}{2})},
\end{equation}
see e.g. \cite[p.~337]{AGM15}, \cite[Lem.~3.12]{Kol05} or \cite[p.~178]{LYZ04}. 

\begin{rem}\label{rem:akq}
We can now justify the final statement in Remark~\ref{rem:finite-rate}. As a special case of the above, if $I_p(\alpha\IB_2^k)<\infty$ for some $\alpha\in (0,1]$, then $\sup_{\mu\in \widetilde{M}_{\alpha}}\Ent(\mu)$ is finite, where the set $\widetilde{M}_{\alpha}=\{\mu\in\cP_q(\IR^k)\colon Z_q(\mu)=\alpha\IB_2^k\text{ and }Z_2(\mu)\subset \IB_2^k\}$ is compact, convex and invariant under rotations. By Lemma~\ref{lem:max-invariant}, we find a maximum entropy distribution $\mu^*\in \widetilde{M}_{\alpha}$ which gives rise to a density $v$ of the above form. It follows from \eqref{eq:radial} that $\alpha^q=a_{k,q}\int_0^{\infty}r^qv(r)\dd r$ and $1\ge a_{k,2}\int_0^{\infty}r^2v(r)\dd r$. Since $a_{k,2}=\frac{1}{k}$, Lyapunov's inequality gives $\alpha<a_{k,q}^{1/q}\sqrt{k}=\overline{a}_{k,q}m_q$, where $m_q$ is as in \eqref{eq:ball-min} and 
\[
\overline{a}_{k,q}
:=\bigg(\frac{\Gamma(\frac{k}{2})}{\Gamma(\frac{k+q}{2})}\bigg)^{1/q}\Big(\frac{k}{2}\Big)^{1/2}
=\bigg(\frac{\Gamma(\frac{k}{2}+1)}{\Gamma(\frac{k}{2}+\frac{q}{2})}\bigg)^{1/q}\Big(\frac{k}{2}\Big)^{1/2-1/q}
>1,
\]
by Gautschi's inequality for the Gamma function. Note that Lyapunov's inequality is strict since $v$ is a density with respect to Lebesgue measure.
\end{rem}

Rewriting the optimization problem $\sup_{\mu\in M_{\beta}}\Ent(\mu)$ in terms of $v$, we want to find a density $v$ on $[0,\infty)$ solving
\begin{align}\label{eq:max-problem}
	\max \qquad &-\int_{0}^{\infty}\log(v(r))v(r)\dd r+\int_0^{\infty}v(r) \log(r^{k-1})\dd r +\log(\omega_k)	\notag\\
\text{subject to }\quad  & \int_{0}^{\infty} r^q v(r)\dd r\le \frac{\beta^q}{a_{k,q}}\notag\\
 & \int_{0}^{\infty} r^2 v(r)\dd r\le k,
\end{align}
since $a_{k,2}=\frac{1}{k}$.  The solution $v^*$ to this optimization problem is known to be of the form
\begin{equation} \label{eq:max-ent}
v^*(r)=\frac{1}{Z}r^{k-1}\exp(-\lambda_1 r^q - \lambda_2 r^2), \qquad r>0,
\end{equation}
where $\lambda_1,\lambda_2\ge 0$ are such that $v^*$ satisfies the constraints in \eqref{eq:max-problem} and $Z$ is a normalizing constant. Moreover, we have the `complementary slackness' conditions
\begin{equation} \label{eq:slackness}
\Big(\int_0^{\infty}r^q v^*(r)\dd r- \frac{\beta^q}{a_{k,q}}\Big)\lambda_1=0\quad\text{ and }\quad
\Big(\int_0^{\infty}r^2 v^*(r)\dd r- k\Big)\lambda_2=0.
\end{equation}
This follows for example from the proof of \cite[Thm.~12.1.1]{CT06} (see also~\cite[Lem.~6]{KR18}). We refer to \cite{KR18} to a related geometric problem giving rise to a similar optimization problem.

Analogously, the optimization problem \eqref{eq:max-problem} without the second constraint is solved by $v^*$ of the form \eqref{eq:max-ent} with $\lambda_2=0$ and $\lambda_1>0$ chosen in such a way that the first constraint is satisfied with equality. The formula
\[
\int_0^{\infty}r^a e^{-b r^c}\dd r
=\frac{\Gamma(\frac{a+1}{c})}{c b^{\frac{a+1}{c}}},\qquad \text{ for } a\ge 0 \text{ and }b,c>0,
\]
yields that the optimal $v^*$ solving \eqref{eq:max-problem} without the second constraint is of the form
\begin{equation}\label{eq:opt-1}
v^*(r)=\frac{1}{Z_{k,q,\beta}}r^{k-1}e^{-\lambda_{k,q,\beta} r^q}
\quad\text{ with } Z_{k,q,\beta}=\frac{\Gamma(\frac{k}{q})}{q\lambda_{k,q,\beta}^{k/q}} \text{ and }\lambda_{k,q,\beta}=\frac{ka_{k,q}}{q \beta^q}.
\end{equation}
Moreover, it holds that
\[
\int_0^{\infty}r^2 v^*(r) \dd r 
=\frac{\Gamma(\frac{k+2}{q})}{\Gamma(\frac{k}{q})}\lambda_{k,q,\beta}^{-2/q}
=\frac{\Gamma(\frac{k+2}{q})}{\Gamma(\frac{k}{q})}\Big(\frac{ka_{k,q}}{q}\Big)^{-2/q}\beta^2,
\]
which is increasing in $\beta$ and equal to $k$ for the choice of $\beta=\beta_{k,q}$ as in Theorem~\ref{thm:small-ball}. Note that $\delta_{k,q}$ as in Theorem~\ref{thm:small-ball} satifies $\delta_{k,q}=\frac{k}{q}a_{k,q}$. In particular, if $\beta\le \beta_{k,q}$, then $ \int_0^{\infty}r^2 v^*(r) \dd r \le k$ and $v^*$ satisfies the constraints in \eqref{eq:max-problem} as well as maximizes entropy over all densities $v$ satisfying the first constraint with equality. It also solves the original problem in \eqref{eq:max-problem} since any solution $v^*$ to \eqref{eq:max-problem} with $\lambda_1,\lambda_2>0$ satisfies, by \eqref{eq:slackness}, both constraints with equality. The corresponding maximum entropy is given by
\[
-\int_{0}^{\infty}\log(v^*(r))v^*(r)\dd r+\int_0^{\infty}v^*(r) \log(r^{k-1})\dd r +\log(\omega_k)
=\frac{k}{q}+\log(Z_{k,q,\beta})+\log(\omega_k).
\]
Thus, we obtain the following partial solution to $\sup_{\mu\in M_{\beta}}\Ent(\mu)$.

\begin{lem}\label{lem:maximizer}
For any $0<\beta\le \beta_{k,q}$, the maximum entropy distribution $\mu^*$ in $M_{\beta}$ has Lebesgue density
\[
f^*\colon \IR^k\to \IR, \qquad x\mapsto\frac{1}{\omega_kZ_{k,q,\beta}} e^{-\lambda_{k,q,\beta} \|x\|^q},
\]
with $Z_{k,q,\beta}$ and $\lambda_{k,q,\beta}$ as above. It holds that
\[
\Ent(\mu^*)
=k\log \beta+\log\Gamma\Big(\frac{k}{q}\Big)+\frac{k}{q}(1-\log\delta_{k,q})-\log q+\log\omega_k,
\]
where $\delta_{k,q}=\frac{k}{q}a_{k,q}$ is as in Theorem~\ref{thm:small-ball}.  Moreover, the set $A_{\beta}$ defined in \eqref{eq:A-beta} is an $I_p$-continuity set, that is, the first equality in \eqref{eq:continuity-set} holds.
\end{lem}
\begin{proof}
The form of the density of the maximum entropy distribution $\mu^*$ in $M_{\beta}$ is a consequence of the above discussion and the identity $f^*(x)=g(\|x\|)=\omega_k^{-1} \|x\|^{1-k}v^*(\|x\|)$, $x\in \IR^k\setminus\{0\}$. We compute 
\begin{align*}
\Ent(\mu^*)
&=\frac{k}{q}+\log(Z_{k,q,\beta})+\log(\omega_k)\\
&=k\log \beta+\log\Gamma\Big(\frac{k}{q}\Big)+\frac{k}{q}\Big(1-\log\frac{ka_{k,q}}{q}\Big)-\log q+\log\omega_k,
\end{align*}
which is continuous and increasing as a function of $\beta\in (0,\beta_{k,q}]$. This gives
\[
\sup_{\mu\in M_{\beta}'}\Ent(\mu)
=\sup_{\gamma<\beta}\sup_{\mu\in M_{\gamma}}\Ent(\mu)
=\sup_{\mu\in M_{\beta}}\Ent(\mu)
=\Ent(\mu^*),
\]
where $M_{\beta}'$ is as in \eqref{eq:inf-ent-2}. This proves the first equality in \eqref{eq:continuity-set} via \eqref{eq:inf-ent-1} and \eqref{eq:inf-ent-2}. 
\end{proof}

We can now set, for $\beta\le \beta_{k,q}$,
\[
c_{k,q,\beta}
:=k\log \beta-\frac{k}{2}\log(2\pi e)+\log\Gamma\Big(\frac{k}{q}\Big)+\frac{k}{q}(1-\log\delta_{k,q})-\log q+\log\omega_k
\]
to obtain, by \eqref{eq:inf-ent-1},
\[
\inf_{K\in A_{\beta}} I_p(K)
=\frac{k}{2}\log(2\pi e)-\Ent(\mu^*)
=-c_{k,q,\beta}.
\]
This completes the proof of Theorem~\ref{thm:small-ball} in the case of $\beta\le \beta_{k,q}$. 

Recall that the standard Gaussian distribution $\gamma^{\otimes k}$ maximizes entropy over the set $\{\mu\in \cP_q(\IR^k)\colon Z_2(\mu)\subset \IB_2^k\}$ with $\Ent(\gamma^{\otimes k})=\frac{k}{2}\log(2\pi e)$. For $\beta\ge m_q$, it follows from $Z_q(\gamma^{\otimes k})=m_q\IB_2^k\subset \beta \IB_2^k$ with $m_q$ as in \eqref{eq:ball-min} that $\gamma^{\otimes k}$ also maximizes entropy over $M_{\beta}$. Thus,
\[
\inf_{K\in A_{\beta}} I_p(K)=0\quad\text{for }\beta\ge m_q,
\]
which proves Theorem~\ref{thm:small-ball} in the case of $\beta\ge m_q$. 

Before we prove Theorem~\ref{thm:small-ball} in the remaining case $\beta\in (\beta_{k,q},m_q)$, let us prove Proposition~\ref{pro:gamma} which implies that indeed $\beta_{k,q}< m_q$ if $q\neq 2$. 

\begin{proof}[Proof of Proposition~\ref{pro:gamma}]
Let $x,y>0$ and set $x=\frac{1}{q}$ and $y=\frac{1}{p}$ with $0<q,p<\infty$. Let $\beta>0$. Consider the problem of finding a density $v$ on $[0,\infty)$ solving
\begin{align}\label{eq:max-problem-1}
	\max \qquad &\Ent(v)\notag\\
	\text{subject to }\quad  & m_q(v)\le \beta\tag{P($q,\beta$)}
\end{align}
where $\Ent(v)=-\int_{0}^{\infty}\log(v(r))v(r)\dd r$ and $m_q(v):=\int_{0}^{\infty} r^q v(r)\dd r$. We say that $v$ is feasible for P($q,\beta$) if it satisfies the constraint $m_q(v)\le \beta$. Similar to before, the solution $v^*$ to the problem P($q,\beta$) is of the form
\begin{equation} 
	v^*(r)=\frac{\lambda^{1/q}}{\Gamma(\frac{q+1}{q})}e^{-\lambda r^q}, \qquad r>0,
\end{equation}
where $\lambda> 0$ is such that $v^*$ satisfies $m_q(v^*)= \beta$. Similar computations as above show that
\[
v_1^*(r)=\frac{1}{q^{1/q}\Gamma(\frac{q+1}{q})\beta^{1/q}} e^{- \frac{r^q}{q\beta}} \text{ and }v_2^*(r)=\frac{1}{p^{1/p}\Gamma(\frac{p+1}{p})}e^{- \frac{r^p}{p}}, \qquad r>0,
\]
solve the problems P($q,\beta$) and P($p,1$), respectively. It holds that
\[
m_p(v_1^*)
=\frac{\Gamma(\frac{p+1}{q})}{\Gamma(\frac{1}{q})}q^{p/q}\beta^{p/q}
\quad\text{and}\quad
m_q(v_2^*)
=\frac{\Gamma(\frac{q+1}{p})}{\Gamma(\frac{1}{p})}p^{q/p}
\]
and we see that $m_p(v_1^*)\le 1$ holds precisely when 
\[
\beta
\le \beta_{p,q}
:=\frac{1}{q}\bigg(\frac{\Gamma(\frac{1}{q})}{\Gamma(\frac{p+1}{q})}\bigg)^{q/p}.
\]
In particular, the solution $v_1^*$ to P($q,\beta_{p,q}$) is feasible for P($p,1$) and satisfies $\Ent(v_1^*)\le \Ent(v_2^*)$. Suppose that $m_q(v_2^*)\le \beta_{p,q}$. Then $v_2^*$ is feasible for P($q,\beta_{p,q}$) and satisfies $\Ent(v_2^*)\le \Ent(v_1^*)$. Consequently, we must have that $\Ent(v_1^*)=\Ent(v_2^*)$ and both $v_1^*$ and $v_2^*$ solve both P($q,\beta_{p,q}$) and P($p,1$). This can only be true if $p=q$ and then in fact $m_q(v_2^*)=\beta_{p,q}$. Thus, if $p\neq q$ we must have $\beta_{p,q}< m_q(v_2^*)$, i.e. that
\[
\frac{\frac{1}{q}\Gamma(\frac{1}{p})}{\Gamma(\frac{q+1}{p})}\bigg(\frac{\frac{1}{p}\Gamma(\frac{1}{q})}{\Gamma(\frac{p+1}{q})}\bigg)^{q/p}< 1,
\]
which is the inequality in Proposition~\ref{pro:gamma} for $x=\frac{1}{q}$ and $y=\frac{1}{p}$ with $x\neq y$. This completes the proof.
\end{proof}

In the case of $\beta\in (\beta_{k,q},m_q)$ the solution $v^*$ to \eqref{eq:max-problem} is of the form \eqref{eq:max-ent} with $\lambda_1,\lambda_2>0$. For if not, then either $\lambda_1$ or $\lambda_2$ were equal to zero and the corresponding constraint in \eqref{eq:max-problem} has to be satisfied. 

We observe that the optimization problem in \eqref{eq:max-problem} without the first constraint is solved by
\begin{equation} \label{eq:opt-2}
v^*(r)=\frac{2}{\Gamma(\frac{k}{2})2^{k/2}}r^{k-1}e^{-r^2/2},
\end{equation}
which corresponds to $\gamma^{\otimes k}$ and also satisfies the first constraint since 
\[
\int_{0}^{\infty}r^q v^*(r)\dd r
=\frac{1}{a_{k,q}}\int_{\IR^k}|\langle x,u\rangle|^q\dd\gamma^{\otimes k}(x)
=\frac{m_q^q}{a_{k,q}}
\le \frac{\beta^q}{a_{k,q}}.
\]
Arguing as above, this choice of $v^*$ solves also the original problem in \eqref{eq:max-problem}.

The solutions corresponding to either $\lambda_1$ or $\lambda_2$ equal to zero are given by \eqref{eq:opt-1} and \eqref{eq:opt-2}, which for $\beta\in (\beta_{k,q},m_q)$ do not satisfy the corresponding other constraint, and are thus not solutions of \eqref{eq:max-problem}. Thus, $\lambda_1$ and $\lambda_2$ must be positive. Again, let $\mu^*$ be the measure with density $x\mapsto g(\|x\|)$, where $g(\|x\|)=\omega_k^{-1} \|x\|^{1-k}v^*(\|x\|), x\in\IR^k\setminus\{0\}$. Then 
\[
\Ent(\mu^*)
=-\int_{0}^{\infty}\log(v^*(r))v^*(r)\dd r+\int_0^{\infty}v^*(r) \log(r^{k-1})\dd r +\log(\omega_k)
\]
is nondecreasing in $\beta$, as this enlarges the set of measures satisfying the constraints, and depends continuously on $\beta$. Therefore, an argument as in the proof of Lemma~\ref{lem:maximizer} yields that $A_{\beta}$ is an $I_p$-continuity set also for $\beta\in (\beta_{k,q},m_q)$, i.e., we have
\[
\inf_{K\in A_{\beta}} I_p(K)
=\frac{k}{2}\log(2\pi e)-\Ent(\mu^*)
=-c_{k,q,\beta},
\]
for some constant $c_{k,q,\beta}<0$. This completes the proof of Theorem~\ref{thm:small-ball}. We do not investigate the value of this constant as this requires a more involved computation.

\begin{rem}
One can also consider computing the asymptotics of $\IP[\alpha\IB_2^k\subset Z_{n,p}]$	as $n\to\infty$ with some $\alpha\in (0,a_{k,q}^{1/q}\sqrt{k})$, where the restriction is due to Remark~\ref{rem:akq}. This lends itself to a similar analysis and leads to problem \eqref{eq:max-problem} with the first constraint replaced by 
\begin{equation} \label{eq:mod-constraint}
\int_{0}^{\infty}r^qv(r)\dd r\ge \frac{\alpha^q}{a_{k,q}}.
\end{equation}
The solution $v^*$ to this modified optimization problem is of the same form as in \eqref{eq:max-ent} but with $\lambda_1\le 0$ and $\lambda_2\ge 0$ such that $v^*$ satisfies the modified constraints. Due to integrability of $v^*$, we must have $\lambda_2>0$ and equality in the second constraint. If $\lambda_1=0$, then we arrive at the density in \eqref{eq:opt-2}, which corresponds to the Gaussian distribution $\gamma^{\otimes k}$ and can only satisfy \eqref{eq:mod-constraint} if $\alpha\le m_q$. In this case, it holds that
\[
\lim_{n\to\infty}n^{-1}\log\IP[\alpha\IB_2^k\subset Z_{n,p}]=0.
\]
If $\alpha\in (m_q,a_{k,q}^{1/q}\sqrt{k})$, then necessarily $\lambda_1<0$ and $\lambda_2>0$ in \eqref{eq:max-ent} and we do not compute the limit since this requires a more involved computation.
\end{rem}

\subsection*{Acknowledgement}

MS is grateful for the hospitality of the Institute of Mathematical Stochastics, University of M\"unster, during a long-term visit. This research was funded in whole or in part by the Austrian Science Fund (FWF) [Grant DOI: 10.55776/P32405; 10.55776/J4777]. ZK was supported by the German Research Foundation under Germany's Excellence Strategy  EXC 2044 -- 390685587, \textit{Mathematics M\"unster: Dynamics - Geometry - Structure} and by the DFG priority program SPP 2265 \textit{Random Geometric Systems}.  For open access purposes, the authors have applied a CC BY public copyright license to any author-accepted manuscript version arising from this submission.

\bibliographystyle{abbrv}
\bibliography{shadows}

\end{document}